\documentclass{article}
\usepackage{amsthm,amssymb,amsmath,enumerate,graphicx, tikz}
\usepackage{amscd}
\usetikzlibrary{decorations.markings}
\usepackage[justification=centering]{caption}
\usepackage{setspace}
\usepackage{comment}
\usepackage{geometry}
\theoremstyle{plain}
\usepackage[square,numbers]{natbib}

\usepackage{caption}

\newtheorem{theorem}{Theorem}[section]
\newtheorem{corollary}[theorem]{Corollary}
\newtheorem{fact}[theorem]{Fact}

\newtheorem{definition}[theorem]{Definition}
\newtheorem{lemma}[theorem]{Lemma}
\theoremstyle{remark}
\newtheorem{remark}[theorem]{Remark}
\newtheorem{challenge}[theorem]{Challenge}
\newtheorem{example}[theorem]{Example}

\newcommand{\R}{\mathbb R}

\newcommand{\refT}[1]{Theorem~\ref{#1}}
\newcommand{\refC}[1]{Corollary~\ref{#1}}

\newcommand{\refL}[1]{Lemma~\ref{#1}}

\newcommand{\refS}[1]{Section~\ref{#1}}

\let\OLDthebibliography\thebibliography
\renewcommand\thebibliography[1]{
  \OLDthebibliography{#1}
  \setlength{\parskip}{0pt}
  \setlength{\itemsep}{0pt plus 0.3ex}
}

\title{Tight infinite matrices}
\author{Ron Aharoni\thanks{Faculty of Mathematics, Technion, Haifa~32000, Israel and MIPT. E-mail: {\tt ra@technion.ac.il}. Research supported by the Israel Science Foundation (ISF) grant no. 2023464 and the Discount Bank Chair at the Technion, and the European Union's Horizon 2020 research and innovation programme under the Marie Skldowska-Curie grant agreement no.\ 823748.}
~and He Guo\thanks{Faculty of Mathematics, Technion, Haifa~32000, Israel. E-mail: {\tt hguo@campus.technion.ac.il}.}}
\date{November 2022}

\begin{document}

\maketitle

\begin{abstract}
    We give a simple proof of a recent result of Gollin and Jo\'o \cite{gollinjoo}: if a possibly infinite system of homogeneous linear equations $A\vec{x}=\vec{0}$, where $A=(a_{i,j})$ is an $I \times J$ matrix, has only the trivial solution, then there exists an injection $\phi: J \to I$, such that $a_{\phi(j),j}\neq 0$ for all $j \in J$.
\end{abstract}

\section{Introduction and preliminaries}

Throughout the paper, $I$ and $J$ are  possibly infinite sets.
\begin{itemize}
\item
Let $FS(J)$ be the set of vectors in $\R^J$ having finite support, and $RFS(I,J)$  the set of $I \times J$ matrices whose rows belong to $FS(J)$. 
\item
The $i$th row of a matrix $M$ is denoted by 
$M_i$, and the $(i,j)$-entry sometimes by $M_{i,j}$ and sometimes (as is common) by $m_{i,j}$. 
\item
The set of rows of a matrix $M$ is denoted by $Row(M)$, and the set of columns by $Col(M)$.

\item For a possibly infinite set~$K$,
let $e_j$ be the $j$th vector in the standard basis of $FS(K)$: $e_j(k)=\delta_j^k$ for any $k\in K$.
\item The kernel $ker(M)$ of   a matrix
$M \in RFS(I,J)$ is $\{\vec{x} \in \R^J \mid M\vec{x}=\vec{0}$\}.
\end{itemize}
\begin{definition}
   A matrix $A \in RFS(I,J)$ is called {\em tight} if $ker(A) =\{\vec{0}\}$, i.e if the system of equations $A\vec{x}=\vec{0}$ has only the trivial solution. 
   \end{definition}
   Note that $A\vec{x}$ is a possibly infinite linear combination of the columns of~$A$. The condition $A \in RFS(I,J)$ implies that the product $A\vec{x}$ is well-defined. So, tightness is stronger than  linear independence of $Col(A)$.
In the finite case the two conditions coincide, and if they hold then $rank(A)=|J|$,  and hence $|J| \le |I|$.

\begin{example}\label{donjuan}
The  $\mathbb{Z}^+ \times \mathbb{Z}^+$ matrix $A=(a_{i,j})$ defined by $a_{i,1}=1$ and $a_{i,i+1}=-1$ for each~$i$, and $a_{i,j}=0 $ for all other values of $i,j$ is column-independent but not tight, since $A\vec{1}=\vec{0}$. 
\end{example}

The  inequality $|J| \le |I|$ is easily seen to be true also in the infinite case --- see \refC{cor:indlessthanspan} below. The theme pursued in the paper is a strengthening of this fact. 
A recurring theme in infinite combinatorics is replacing inequalities between sizes by injections satisfying a condition pertaining to the setting. A famous case is Erd\H{o}s'  conjecture (proved in \cite{ab}) on the infinite version of Menger's theorem: given any two sets $A,B$ of vertices in a digraph there exist a set $F$ of disjoint $A-B$ paths and an $A-B$ separating set $S$, such that $S$ consists of the choice of one vertex from each path in $F$. This implies $|S|=|F|$, but is more specific. In our case  the pertinent condition is summarized in:

\begin{definition}
An $I \times J$ matrix $A=(a_{i,j})$ is  {\em loaded} if there exists an injection $\phi: J \to I$ such that   that $a_{\phi(j),j}\neq 0$ for all $j\in J$.
\end{definition}

Gollin and Jo\'o \cite{gollinjoo} proved:

\begin{theorem}\label{thm:main}
A  tight matrix is loaded.  \end{theorem}
The matrix in example \ref{donjuan} is column-independent,  and not loaded. So, the tightness condition is indeed necessary, column-independence does not suffice. 

The finite case of  Theorem \ref{thm:main} is a special case of a well-known property of matroids:

\begin{fact}\label{matroids}
Let $\mathcal M$ be a matroid with finite rank. Let $B$ be a base of $\mathcal{M}$
and $H$ an independent set. Then there exists an injection $f:H \to B$ such that $f(h)$ belongs to the minimal subset of $B$ spanning $h$ for each $h\in H$. 
\end{fact}
One way to prove it uses Hall's theorem. For every subset $T$ of $H$ let $U_T$ be
the union, over $t \in T$, of the minimal  subsets of $B$ spanning $t$. Then, since $T \in {\mathcal M}$, we have  $|U_T| \ge |T|$. 

To derive the finite case of \refT{thm:main}, take the ground set of $\mathcal{M}$ to be  $\mathbb{R}^I$ and the independent sets to be the sets of independent vectors in~$\mathbb{R}^I$, $B$ to be $\{e_i\in \mathbb{R}^I \mid i\in I\}$
and $H$ to be $Col(A)$. 
\begin{remark}
    The infinite case does not follow this way, because  $B$ is a base for $FS(I)$,  not for $\R^I$, the habitat of $Col(A)$.
\end{remark}
Another proof of the finite case of \refT{thm:main} uses the fact that $rank(A)$ is the largest $k$ for which there exists a $k \times k$ submatrix with  non-zero determinant.

The   proof in \cite{gollinjoo} for the infinite case uses, in a quite ingenious way, a criterion for matchability (a ``marriage theorem'') by Wojciechowski \cite{w}, a refined version of a theorem of Nash-Williams. We shall later  return to this approach.   The main aim of this paper is to provide a direct, simpler proof. On the way we shall also mark some  facts about infinite linear spaces.

We shall use a compactness theorem of Cowen and Emerson~\cite{CowenEmerson}.
\begin{theorem}\label{thm:compactness}
    If every finite subset of a set $T$ of linear equations is solvable, then
so is $T$.
\end{theorem}

\section{Conditions  equivalent to tightness }

Call a variable $x_j$ \emph{stubborn} if there is no solution of the linear system $A\vec{x}=\vec{0}$ in which $x_j=1$.
 Tightness is equivalent to all variables being stubborn.

\begin{lemma}\label{lemma:stubborn}
 A variable $x_j$ is stubborn if and only if $e_j \in span(Row(A))$ for $e_j\in \mathbb{R}^J$. 
\end{lemma}

Here, as usual, for a set~$S$ of vectors in a linear space, $span(S)$ is the set of finite linear combinations of  vectors in $S$.
\begin{proof}
The ``if'' part is clear.
To prove the ``only if'' direction, assume $x_j$ is stubborn. Add to the system of equations the equation $x_j=1$. By assumption the new system of equations is non-solvable, and by \refT{thm:compactness} it has a non-solvable finite subsystem $I'$.     Gauss elimination produces then a row $(0,0,0,\dots|1)$, where as usual the ``$|$'' separates the row of $A$ from the right-hand side of the equation. This implies the existence of coefficients $\lambda_i,~i \in I' $ such that  $\sum_{i \in I'}\lambda_i(A_i|0)+\lambda (e_j |1)=(0,0,0,\ldots |1)$ (whence $\lambda=1$). Then  $e_j =-\sum_{i\in I'}\lambda_iA_i$. 
 \end{proof}

\begin{theorem}
       Let $A \in RFS(I,J)$.  The following are equivalent: 

    \begin{enumerate}[(1)]
        \item $A$ is tight.\label{eq:tight}
        \item $span(Row(A))=FS(J)$. \label{eq:span=FS}
        \item $A$ is left-invertible, i.e., there exists an $J \times I$ matrix $Z \in RFS(J,I)$  such that $(ZA)_{\ell,k}=\delta^\ell_k$ ($Z \in RFS(J,I)$ implies that $ZA$ is well-defined). \label{eq:leftI}
    \end{enumerate}
\end{theorem}

\begin{proof} 
   \eqref{eq:tight} $\Rightarrow$ \eqref{eq:span=FS}: 
Since $A\in RFS(I,J)$,  $span(Row(A))\subseteq FS(J)$. By~\refL{lemma:stubborn}, if $A$ is tight, then $e_j \in span(Row(A))$ for every $e_j\in\mathbb{R}^J$ with $j \in J$, proving $FS(J)\subseteq span(Row(A))$.  

\eqref{eq:span=FS} $\Rightarrow$ \eqref{eq:leftI}: By~\refL{lemma:stubborn}, for each $j\in J$ we have $e_j\in FS(J)=span(Row(A))$, therefore there exist finitely many non-zero coefficients $\lambda_{j,i_1},\dots, \lambda_{j, i_{t(j)}}$ such that 
$\lambda_{j, i_1}A_{i_1}+\dots+ \lambda_{j, i_{t(j)}}A_{i_{t(j)}}=e_j$.
The desired matrix $B$ is defined by 
$B_{j, p}= \lambda_{j,p}$ for each  $j\in J$ and $ p=i_1,\dots, i_{t(j)}$, and all other entries being $0$.

\eqref{eq:leftI} $\Rightarrow$ \eqref{eq:tight}: 
Assume $(ZA)_{\ell,k}=\delta^\ell_k$.
Suppose $\vec{x}$ is a non-trivial solution to $A\vec{x}=\vec{0}$ with $x_j\neq 0$. Then  $ZA\vec{x}=\vec{0}$, implying $(ZA)_j\vec{x}=\vec{0}$ so that $x_j=0$, a contradiction.
\end{proof}

\begin{corollary}\label{cor:indlessthanspan}
If $A\in RFS(I,J)$ is tight then $|I| \ge |J|$.
\end{corollary}
We need the following basic fact, whose proof is given for completeness. 
\begin{lemma}\label{lemma: LleS}
    If $S,L \subseteq FS(J)$, $span(S)=FS(J)$ and $L$ is linearly independent then $|L|\le |S|$.
\end{lemma}
\begin{proof}
    If $S$ is finite, then this is a classical theorem on finite vector spaces. If not, then for every $v \in L$ let $\Phi(v)$ be a finite subset of $S$ spanning $v$. By the finite case, $|\Phi^{-1}(T)| \le |T|$ for every finite subset $T$ of $S$. Thus $\Phi$ is a finite-to-one function from  $L$ to $\binom{S}{<\infty}$, the set of finite subsets of $S$, implying $|L| \le |\omega|\cdot|\binom{S}{<\infty}|=|S|$.
\end{proof}

\begin{proof}[Proof of \refC{cor:indlessthanspan}]
By \eqref{eq:span=FS},  $Row(A)$ spans $FS(J)$.
Apply~\refL{lemma: LleS} to the space $FS(J)$ with  $L=\{e_j\in \mathbb{R}^J \mid j\in J \}$ and  $S=Row(A)$. 
\end{proof}

 An $I \times J$ matrix $A$, where $J \subseteq I$, is said to be {\em diagonal} if $a_{i,j}=0$ whenever $i \neq j$. A diagonal matrix is said to be {\em proudly diagonal}  if it is diagonal and $a_{j,j}\neq 0$ for all $j \in J$.

An {\em elementary row operation} on a matrix $A$ is either exchanging two rows;  or replacing a row $A_i$ by a linear combination of finitely many rows $A_k$, in which $A_i$ appears with non-zero coefficient.

In the remaining part of this section, we assume that $I=\mathbb{Z}^+$. 
A matrix $A \in RFS(I,J)$ is called \emph{reducible} to a matrix $B \in RFS(I,J)$ if there is a sequence of matrices $B^{(0)}=A,B^{(1)},B^{(2)},\dots \in RFS(I,J)$ such that for each $k\in I$, $B^{(k)}$ has the same first $k$ rows as $B$, and $B^{(k)}$ is obtained from $B^{(k-1)}$ by finitely many elementary row operations.

 A matrix~$A$ is {\em proudly row-diagonalizable} if $A$ is reducible to  
a proudly diagonal matrix.
 In the countable case there is another condition  equivalent to tightness.

\begin{theorem}\label{thm:tightiffprd}
    Let $A \in RFS(I,J)$, where $I,J$ are countable and  $J\subseteq I$. Then $A$ is tight if and only if $A$ is  proudly row-diagonalizable.
\end{theorem}
\begin{proof}[Proof of the ``if" part]
Suppose that
 $A$ is proudly row-diagonalizable, i.e., $A$ is reducible to a proudly diagonal matrix $B$, and that  $A$ is not tight, meaning that there exists $\vec{x} \in \R^J$ such that $A\vec{x}=\vec{0}$ and $x_j\neq 0$ for some $j \in J$.
Then by the definition of reducibility, 
using finitely many elementary row operations, $A$ can be reduced  to a matrix $B^{(j)}$ whose first $j$ rows are the same as those of the proudly row-diagonal matrix~$B$. Therefore $B^{(j)}\vec{x}=\vec{0}$ and in particular $B^{(j)}_j\vec{x}=\vec{0}$. Since $B^{(j)}_{j,j}$ is the only non-zero entry in $B^{(j)}_j$,  $B^{(j)}_j\vec{x}=\vec{0}$ means $B^{(j)}_{j,j}x_j=0$, which implies $x_j=0$, a contradiction. 
\end{proof}
We defer the proof of the ``only if" part to \refS{subsubsec:theonlyifpart}.
 
\section{Proof of \refT{thm:main}}\label{sec:proof}

\subsection{The countable case}

We construct   an injective function $\phi: J \to I$ satisfying  $a_{\phi(j), j}\neq 0$ for all $j\in J$. The values $\phi(j)$ are defined inductively.

We start with $j=1$. By~\refL{lemma:stubborn} the stubbornness of~$x_1$ implies 
\[e_1=\sum_{i \in I_1}\lambda^{(0)}_iA_i\] for a finite set $I_1 \subseteq I$.
Let $i_1 \in I_1$ satisfy  $\lambda^{(0)}_{i_1}\neq 0$. Define 
 $\phi(1)=i_1$. 
By permuting rows we may (and do) assume $i_1=1$.

Let $A^{(0)}=A$ and let $A^{(1)}=(a^{(1)}_{m,n})$ be $A$ with
$A_1$ replaced by $e_1$. The fact  $\lambda^{(0)}_{1}\neq 0$ means that $A_1$ is a linear combination of $e_1$ and other rows in~$I_1$,
hence 
\begin{equation}\label{eq:equivalent}
    span(Row(A^{(1)}))= span(Row(A^{(0)}))=span(Row(A)).
\end{equation}

Next  consider $j=2$. Since~$x_2$ is stubborn, by~\refL{lemma:stubborn} and~\eqref{eq:equivalent}, we have $e_2 \in span(Row(A))=span(Row(A^{(1)}))$ so that 
\[   e_2 = \sum_{i\in I_2}\lambda^{(1)}_i A^{(1)}_i  \]
for a finite set $I_2\subseteq I$.
Noting that $A^{(1)}_1=e_1$, there exists $i_2\neq 1$ such that $\lambda^{(1)}_{i_2}\neq 0$ and $a^{(1)}_{i_2,2}\neq 0$. Define $\phi(2)=i_2$. Permuting rows, we may assume $i_2=2$. Note that in the original matrix~$A$, $a_{i_2,2}\neq 0$ as $A^{(1)}_\ell=A_\ell$ for all $\ell\neq 1$. So $\phi$ is injective on $\{1,2\}$ and satisfies the requirement of~\refT{thm:main}. We replace the $i_2$th row of $A^{(1)}$ by $e_2$ to obtain $A^{(2)}=(a^{(2)}_{m,n})$. We have
\[  span(Row(A^{(2)}))=span(Row(A^{(1)}))=span(Row(A)). \]
By construction, $A^{(2)}_\ell$ is $e_\ell$ for $\ell\in\{1,2\}$ and is $A_\ell$ for $\ell>2$.

Assume that we have defined an injective function $\phi$ on $\{1,\dots, j \}$ satisfying the requirement of~\refT{thm:main}, and a matrix $A^{(j)}=(a^{(j)}_{m,n})$ satisfying (i)~$A^{(j)}_\ell = e_\ell$ for all $ \ell\in\{1,\dots,j\}$ and $A^{(j)}_\ell = A_\ell$ for all $\ell>j$, and (ii)~~ 
\[ span(Row(A^{(j)}))=span(Row(A)).  \]

Next look at  $j+1$. By  by~\refL{lemma:stubborn}
the stubbornness of~$x_{j+1}$ implies  $e_{j+1}\in span(Row(A))=span(Row(A^{(j)}))$, so there exists a finite set $I_{j+1}\subseteq I$ satisfying 
\begin{equation}\label{eq:recurrsiverelation}
    e_{j+1}=\sum_{i\in I_{j+1}}\lambda^{(j)}_i A^{(j)}_i.
\end{equation}
There exists $i_{j+1}\not\in \{1,\dots, j\}$ satisfying $\lambda^{(j)}_{i_{j+1}}\neq 0$ and $a^{(j)}_{i_{j+1},j+1}\neq 0$. Setting $\phi(j+1)=i_{j+1}$ satisfies the desired property in~\refT{thm:main} since $A^{(j)}_{i_{j+1}}=A_{i_{j+1}}$. Then $a_{i_{j+1},j+1}=a^{(j)}_{i_{j+1},j+1}\neq 0$.
Permuting rows, we may assume $i_{j+1}=j+1$.
Replacing  $A^{(j)}_{i_{j+1}}$ by $e_{j+1}$ to obtain $A^{(j+1)}$, we have 
\[  span(Row(A^{(j+1)}))=span(Row(A^{j}))=span(Row(A)). \]
Furthermore, $A^{(j+1)}_\ell$ is $e_\ell$ for $\ell \in\{1,\dots, j+1\}$ and is $A_\ell$ for $\ell>j+1$.

Continuing, we obtain the desired injection $\phi$.

\subsubsection{Completing the proof of~\refT{thm:tightiffprd}}\label{subsubsec:theonlyifpart}
\begin{proof}[Proof of the ``only if" part]
    Assume $A$ is tight. The proof is similar to that of~\refT{thm:main}. With the same notation as above, assume that for some $j\ge 0$,  $A$ is reducible to a matrix $A^{(j)}$ satisfying  $A^{(j)}_{p,q}=\delta_{p}^q$ for any $p\le j$ and any $q$. For~$j+1$, we have
\begin{equation*}
    e_{j+1}=\sum_{i \in I_{j+1}} \lambda_{i}^{(j)}A^{(j)}_i.
\end{equation*}
We choose the row~$A^{(j)}_{\phi(j+1)}$, replace it by the linear combination $\sum_{i\in  I_{j+1}} \lambda_{i}^{(j)}A^{(j)}_i=e_{j+1}$ (note that $\lambda_{\phi(j+1)}^{(j)}$, the coefficient of~$A^{(j)}_{\phi(j+1)}$, is non-zero), and swap it with the $(j+1)$st row. So far,  only finitely many elementary row operations have been used. Name the resulting matrix~$A^{(j+1)}$. Then $A^{(j+1)}$ has the same first~$j$ rows as~$A^{(j)}$ (since $\phi(j+1)\not\in\{1,\dots, j\}$) and has $e_{j+1}$ as its $(j+1)$st row. Doing it inductively produces a proudly row-diagonal matrix. 
\end{proof}

\subsection {The general (possibly uncountable) case}\label{general}
The  case of uncountable $J$ is proved using  the fact that $A \in RFS(I,J)$. This enables choosing the variables $x_j$ in the proof above in such an order that there is an ordinal $\alpha \le \omega$, satisfying the following:  

All variables appearing in equations $A_i\vec{x}=0, i \in I_j$, $j<\alpha$, appear as $x_k$ for some $k<\alpha$. (Here $I_j$ is as in the proof above.)

Removing these equations and variables from the system results then in a tight system, and we can continue inductively. 

This argument will be repeated  in the simpler setting of matchings in  \refS{sec:marriage}.

\section{Marriage theorems}\label{sec:marriage}

\begin{definition}

A bipartite graph with sides $(M,W)$  is called {\em espousable} if it contains a matching covering $M$.
\end{definition}

For a matrix $A=(a_{i,j})_{i \in I, j \in J}$, let $G_A$ be the bipartite graph with sides $(J,I)$ and edges set $E$ defined by $(j,i) \in E$ if $a_{i,j} \neq 0$. \refT{thm:main} can be re-formulated as:

\begin{theorem}
    If $A$ is tight then $G_A$ is espousable.
\end{theorem}
Necessary and sufficient conditions for espousability (these are called ``marriage theorems'') are known in general graphs, but  in our context  only  criteria for countable graphs are needed. In this case, there are basically two criteria known. One is a refinement by Wojchiechowsky of a criterion proved by Nash-Williams. It is the absence of a substructure that we shall name a ``NWW-obstruction''. We shall not define it here, since it is rather involved, and not directly relevant in this paper. One can find it e.g. in \cite{gollinjoo}. In \cite{w} it was proved that this criterion applies also in  bipartite graphs having countable degrees in the ``women'' ($W$) side.

The other criterion was proved by Podewski and Steffens. It is the absence of another type of sub-structure, that we shall name a ``PS-obstruction''. 
 In \cite{a} it was proved that the two criteria are equivalent in general, namely the existence of a PS obstruction  is equivalent to the existence of an NWW-obstruction, regardless of countability. 

Here is the definition of a PS-obstruction.
 For a set $T$ of vertices let $N(T)=N_G(T)$ be the set of neighbors of $T$ in $G$. When~$v$ is a single vertex, we abbreviate $N(\{v\})$ by $N(v)$.
\begin{definition}\hfill
\begin{enumerate}[(1)]
    \item A matching $F$ is called
a {\em wave} if 
     $$N(\bigcup F \cap M)= \bigcup F \cap W.$$ 
\item A wave is called {\em critical} if for every matching $K$ of $\bigcup F \cap M$ we have $\bigcup K \cap W =\bigcup F \cap W$.
\item A pair $(F,a)$ is called an {\em impediment} if $F$ is a wave,  $a \in M \setminus \bigcup F$ and $N(a) \subseteq \bigcup F$.
\item An impediment $(F,a)$ is called a 
{\em Podewski--Steffens (PS for short) obstruction}  if $F$ is critical.  
\end{enumerate}
    
\end{definition}

Clearly, a finite  wave is critical, and a finite impediment is an obstruction. The existence of an obstruction excludes espousability, the existence of an impediment not necessarily.  
It is easy to see that, given condition (1), condition (2) is equivalent to the absence of an infinite $F$-alternating path starting with a non-$F$ edge.

A graph not containing a PS-obstruction is called 
{\em unobstructed}. Clearly,  if a graph is espousable, then it is unobstructed. Podewski and Steffens \cite{ps} proved the other direction: 
\begin{theorem}
If $M$ is countable and the graph is unobstructed then it is espousable.
\end{theorem}
This directly follows from the following lemma.
\begin{lemma}\cite{ps}\label{lemma:recurbounti}
If $G $ is unobstructed then for every $m \in M$ there exists $w \in N(m)$ such that $G-m-w$ is unobstructed. 
\end{lemma}
In fact, the lemma entails the stronger: 
\begin{theorem}
If $N(w)$ is countable for every $w \in W$ and the graph is unobstructed, then it is espousable.
\end{theorem}

\begin{proof}
The proof mimics that in \refS{general}. Take $m_1 \in M$. By ~\refL{lemma:recurbounti} there exists $w_1 \in N(m_1)$ for which   $G_1:=G-m_1-w_1$ is unobstructed, and define $\phi(m_1)=w_1$.  Next apply the lemma to $G_1$, with $m_2\in N(w_1)$,

Let $X_1=N(w_1)$. If $X_1 \neq \{m_1\}$, choose $m_2\neq m_1 \in X_1$, and choose $w_2 \in N_{G_1}(m_2)$ for which $G_2:=G_1-m_2-w_2$ is unobstructed.  

Continuing this way to choose $m_i$ and $w_i$. For some $\alpha \le \omega$ we shall have 
 $\{m_j \mid j \le \alpha\} =\bigcup_{i \le \alpha} N(w_i)$, namely all $N(w_i)$ have been represented. The remaining elements of $M$ are connected only to elements in $W\setminus \{w_i:i\le \alpha\}$, and since the remaining graph is unobstructed, we can start the procedure anew. 
\end{proof}

Theorem \ref{thm:main} implies that if $A$ is tight then $G_A$ does not contain a PS-obstruction. In \cite{gollinjoo} it was proved directly that if $A$ is tight then there is no NWW obstruction, which by the result of \cite{a} implies that it does not contain a PS-obstruction. Since PS-obstructions are simpler the following may be of interest: 
\begin{challenge}
Prove directly that if $G_A$ contains a PS-obstruction then $A$ is not tight.
\end{challenge}

The following example shows that the existence of an impediment does not suffice for this purpose, which means that 
 the condition of ``no infinite 
alternating path in the impediment'' must be invoked.

\begin{example}
Consider the system whose $(2k-1)$st equation is $x_{2k-1}+x_{2k}+x_{2k+1}=0$ and the $2k$th equation is $x_{2k}+x_{2k+1}=0$ for every $k=1,2,\ldots$ Note that $x_{j+1}$ shows up in the $j$th equation for each $j=1,2,...,$ therefore together with $x_1$ they form an impediment. But the system is tight, as the $(2k-1)$ and $2k$th equations imply that all the odd-numbered variables are zero. Then the even-numbered equations imply that all the even-numbered variables are zero. 
\end{example}

\bigskip{\noindent\bf Acknowledgements.} 
We are grateful to Dani Kotlar for a useful comment.

 \end{document}